\documentclass{amsart}

\usepackage{esint}
\usepackage{amssymb}
\usepackage[foot]{amsaddr}
\usepackage{tikz}
\usepackage{graphicx}
\usepackage{color,enumerate}
 \usepackage[colorlinks=false]{hyperref}

\newtheorem{theorem}{Theorem}
\newtheorem{lemma}[theorem]{Lemma}
\newtheorem{corollary}[theorem]{Corollary}

\newtheorem{proposition}[theorem]{Proposition}
\theoremstyle{definition}

\newtheorem{remark}[theorem]{Remark}

\newtheorem{definition}[theorem]{Definition}

\newcounter{assumption}

\newcommand{\eref}[1]{(\ref{e.#1})}

\newcommand{\real}{\mathbb{R}}

\newcommand{\grad}{\nabla}
\newcommand{\e}{\varepsilon}
\newcommand{\Tr}{\textup{Tr}}
\newcommand{\Id}{\textup{Id}}
\newcommand{\indicator}{{\bf 1}}

\def\XXint#1#2#3{{\setbox0=\hbox{$#1{#2#3}{\int}$ }
\vcenter{\hbox{$#2#3$ }}\kern-.6\wd0}}

\DeclareMathOperator*{\argmin}{\arg\!\min}
\def\argmin{\mathop{\arg\,\min}\limits}%

\numberwithin{equation}{section}

\begin{document}
\title[Stability and Dynamic Stability of Serrin's Problem]{Stability of Serrin's Problem and Dynamic Stability of a Model for Contact Angle Motion}
\author{William M. Feldman}
\email{feldman@math.uchicago.edu}
\address{Department of Mathematics, The University of Chicago, Chicago, IL 60637, USA}
\thanks{W. M. Feldman partially supported by NSF-RTG grant DMS-1246999.}
\maketitle
\begin{abstract}
We study the quantitative stability of Serrin's symmetry problem and it's connection with a dynamic model for contact angle motion of quasi-static capillary drops. We prove a new stability result which is both linear and depends only on a weak norm
\[ \big\||Du|^2- 1\big\|_{L^2(\partial \Omega)}. \]
This improvement is particularly important to us since the $L^2(\partial \Omega)$ norm squared of $|Du|^2-1$ is exactly the energy dissipation rate of the associated dynamic model.  Combining the energy estimate for the dynamic model with the new stability result for the equilibrium problem yields an exponential rate of convergence to the steady state for regular solutions of the contact angle motion problem.  As far as we are aware this is one of the first applications of a stability estimate for a geometric minimization problem to show dynamic stability of an associated gradient flow.
\end{abstract}
\section{Introduction}\label{sec: Droplet Problem}
We consider the solutions of the following free boundary problem, for $(x,t) \in \real^N \times [0,\infty)$,
\begin{equation}\label{e.dropletprob}\tag{\textup{P}}
\left\{
\begin{array}{lll}
-\Delta u(x,t) = \lambda(t) & \hbox{ in } & \Omega_t(u) = \{u(\cdot,t)>0\} \vspace{1.5mm}\\
\mathcal{V}_n = \tfrac{\partial_t u}{|Du|} = F(|Du|) & \hbox{ on } & \Gamma_t(u) = \partial \Omega_t(u),
\end{array}\right.
\end{equation}
where $\mathcal{V}_n$ means the normal velocity of $\Gamma_t(u)$ and $\lambda(t)$ is a Lagrange multiplier enforcing the volume constraint,
\[ \int u(\cdot,t) \ dx = \text{Vol} \ \hbox{ for all } \ t>0.\]
The above problem is formulated under the assumption that $\Omega_t(u)$ remains connected, in general each connected component should have its own Lagrange multiplier $\lambda(t)$ and the multipliers could change discontinuously at times when disjoint components merge.  

Problem \eref{dropletprob} can be posed entirely in terms of the domain $\Omega_t$.  Fixing the domain $\Omega_t = \Omega$ the solution $u_\Omega$ and the Lagrange multiplier $\lambda(\Omega)$ are determined by,
\begin{equation}\label{e.ulambdadef}
\begin{array}{c}
u_\Omega := \argmin \{ \int_\Omega |Dv|^2dx: v \in H^1_0(\Omega), \int_\Omega v \ dx = \text{Vol} \} \\ \\
\lambda(\Omega) := \min \{ \int_\Omega |Dv|^2dx: v \in H^1_0(\Omega), \int_\Omega v \ dx = \text{Vol} \}
\end{array}
\end{equation}
The energy $\lambda(\Omega)$, when $\textup{Vol} = 1$, is sometimes called the torsional rigidity of the domain $\Omega$. 

This problem is a quasi-static model for the contact angle driven motion of a liquid droplet on a solid surface.  The model is derived under the assumption of small volume or small magnitude $Du$ so that the surface area of the graph can be replaced by the Dirichlet energy.  The quasi-static assumption is that the time scale for the trend to equilibrium of the droplet profile $u$ is much smaller than the time scale for the motion of the contact line $\Gamma_t(u)$.  For more information about the derivation of the model see \cite{Greenspan:1976aa,greenspan}.  The problem can be exactly solved in $N=1$, we will take $N \geq 2$ from now on.  The physically relevant dimensions are $N=1,2$ and the problem can be exactly solved in $N=1$. We will take $N \geq 2$ from now on.

The function $F: [0,\infty) \to \real $ determines the normal velocity of the free boundary based on the contact angle of the graph $(x,u(x,t))$ with the surface $(x,0)$.  $F$ is always assumed to satisfy
\begin{equation}\label{e.F hyp}
 \hbox{$F$ is smooth, $F'>0$, and $F(1)=0$.}
\end{equation}
 The monotonicity implies that the problem \eref{dropletprob} has a formal comparison principle when $\lambda(t)$ is a given function of time, although we emphasize that there is not a comparison principle, at least in the standard sense, for the volume constrained problem.  With the condition $F(1)=0$, the free boundary problem \eref{dropletprob} can be thought of, formally, as a gradient flow for the energy,
\begin{equation}\label{e.Jdef}
\begin{array}{c}
\mathcal{J}(\Omega) := \int_\Omega |Du_{\Omega}|^2 dx + |\Omega|  = \lambda(\Omega) \textup{Vol} + |\Omega|
\end{array}
\end{equation}
in an appropriate metric on bounded subsets of $\real^N$.  

To motivate the gradient flow structure one can compute the following energy decay estimate for smooth solutions of \eref{dropletprob},
\[
\frac{d}{dt}\mathcal{J}(\Omega_t) = \int_{\Gamma_t} (1-|Du|^2)F(|Du|) \leq 0,
\]
in the case $F(|Du|) = |Du|^2-1$ this has a particularly appealing form,
\begin{equation}\label{e.dissipation}
 \frac{d}{dt}\mathcal{J}(\Omega_t) = -\int_{\Gamma_t} (|Du|^2-1)^2.
 \end{equation}
 See below in Section~\ref{sec: exponential rate} for the full computation.  Thus, at least at a formal level, solution of \eref{dropletprob} are driven by the energy dissipation to stationary solutions,
 \begin{equation}\label{e.serrin}\tag{S}
\left\{
\begin{array}{lll}
-\Delta u = \lambda(\Omega) & \hbox{ in } & \Omega(u), \vspace{1.5mm}\\
|Du| = 1 & \hbox{ on } & \Gamma(u), \vspace{1.5mm}\\
\int_{\Omega} u \ dx = \textup{Vol}
\end{array}\right.
\end{equation}
This overdetermined boundary value problem was originally studied by Serrin \cite{Serrin71}, and slightly afterwards with a different approach by Weinberger \cite{Weinberger:1971aa}.  They proved that when $\partial \Omega$ is $C^2$ then, modulo a translation,
\[ \Omega = B_{r_*}  \]
with
\[ r_*(\textup{Vol})^{N+1} = \omega_N^{-1}(N+2)\textup{Vol} \ \hbox{ and } \ \lambda(B_{r_*},\textup{Vol}) = N\left(\tfrac{\omega_N}{N+2}\right)^{\frac{1}{N+1}}\textup{Vol}^{-\frac{1}{N+1}}.\]
The goal of this paper is to prove a quantitative stability estimate for Serrin's symmetry result.  More precisely we would like to establish an estimate which controls the excess energy,
\[ \mathcal{J}(\Omega) - \mathcal{J}(B_{r_*}) \]
by the energy dissipation,
\[ \int_{\partial \Omega} (|Du|^2 - 1)^2 d \sigma .\]
The ideal scenario, from the perspective of the energy dissipation estimate, would be to establish the following inequality,
\begin{equation*}
 \mathcal{J}(\Omega) - \mathcal{J}(B_{r_*}) \leq C \int_{\partial \Omega} (|Du|^2 - 1)^2 d \sigma
 \end{equation*}
with the constant $C$ depending only on quantities which can be controlled by the available a-priori estimates for the flow \eref{dropletprob}.  Then the energy decay would yield,
\begin{equation*}
 \frac{d}{dt}[\mathcal{J}(\Omega_t) - \mathcal{J}(B_{r_*})]= -\int_{\Gamma_t} (|Du|^2-1)^2 \leq - C[\mathcal{J}(\Omega_t) - \mathcal{J}(B_{r_*})]
 \end{equation*}
establishing the expected exponential rate of convergence of the energy,
\[ \mathcal{J}(\Omega_t) - \mathcal{J}(B_{r_*}) \leq [\mathcal{J}(\Omega_0) - \mathcal{J}(B_{r_*})] e^{-Ct}.\]
This is the result we establish, conditional on $\Omega_t$ being a global in time regular solution.  Such regularity is expected to hold for solutions with initial data $\Omega_0$ close to $B_{r_*}$ in a Lipschitz distance.  The regularity hypothesis will be discussed further below in Section~\ref{sec: hypothesis}

\subsection{Stability results for the Serrin Problem}  First we will discuss the stability estimate of Serrin's symmetry problem \eref{serrin}.  Before describing the results of this paper we will discuss the existing literature on this problem.

As mentioned before the radial symmetry of solutions to \eref{serrin} was first proved by Serrin~\cite{Serrin71}.   His proof used the method of moving planes, used before by Alexandrov~\cite{Alexandrov} to show that constant mean curvature hypersurfaces in $\real^N$ are spheres.  Slightly afterwards Weinberger~\cite{Weinberger:1971aa} discovered a very short proof also based on a maximum principle type argument.  More recently there has been interest in the stability estimates of Serrin's problem.  Aftalion, Busca and Reichel~\cite{Aftalion:1999aa} made the moving planes method quantitative and proved a stability estimate in Hausdorff distance for $C^{2,\alpha}$ domains, basically they obtain a logarithmic stability estimate,
\[ R - r \leq C|\log\||Du| - 1\|_{C^1(\partial \Omega)}|^{-1/N},\]
where $r,R$ are respectively the in-radius and the out-radius of $\Omega$ and $u=u_\Omega$ is the torsion solution defined in \eref{ulambdadef}.  Since then the moving planes based method has been studied further by Ciraolo, Magnanini \cite{Ciraolo:2014aa} and Ciraolo, Magnanini and Vespri \cite{Ciraolo:2016aa} obtaining linear stability in terms of the Lipschitz semi-norm,
\[ R - r \leq C [|Du|]_{\partial \Omega} \ \hbox{ with } \ [f]_{\partial \Omega} = \sup_{\substack{x \neq y \\x,y \in \partial \Omega}}\frac{|f(x)-f(y)|}{|x-y|}\] 
for $C^{2,\alpha}$ domains.  A different approach to the symmetry problem, and to the stability estimates, was developed by Brandolini, Nitsch, Salani and Trombetti in several papers \cite{BNST08,BNST-alt,BNST09} where they also found applications of their method to symmetry problems for $k$-Hessian type equations.  The advantage of their method is that it relies more on integration by parts identities than the maximum principle and thus attains estimates in a weaker norm.  The result of \cite{BNST08} gives the following H\"{o}lder stability estimate which, at least in one direction, depends only on the $L^1(\partial \Omega)$ norm, suppose that,
\[\|Du\|_{C^0(\partial \Omega)} \leq 1+ \delta \ \hbox{ and } \ \||Du| - 1\|_{L^1(\partial \Omega)} \leq \delta |\partial \Omega|,\]
then there is a finite collection of balls $\{B_i\}_{i=1}^k$ such that,
\[ ||B_i| - |B_{r_{*,k}}|| \leq C \delta^{\beta} \ \hbox{ and } \ |\Omega \Delta \cup_i B_i | \leq C  \delta^{\beta} \ \hbox{ with } \ \beta = \tfrac{1}{4N+13}.\]
Note that with only the measure theoretic bound it is possible in $N\geq 3$ for $\Omega$ to be close to a union of finitely many balls of radius close to $r_{*,k} = r_*(\textup{Vol}/k)$ joined perhaps by long thin tentacles.  In $N=2$ this should also be possible, but only with adjacent balls connected by very short necks.  We will avoid both of these issues with our regularity assumption, anyway such configurations do not seem to be relevant to the study of the dynamic problem \eref{dropletprob}.  

Our contribution, which takes the approach of \cite{BNST08} as a starting point, is a stability estimate which depends only on a weak norm, the $L^2(\partial \Omega)$ norm of $|Du|^2-1$, and has linear order.  Our result is still perhaps not optimal in terms of the regularity assumption on $\partial \Omega$, but it is optimal in the sense that it allows to prove the exponential convergence rate for \eref{dropletprob}.  

Our result is also connected with a similar kind of stability estimate for hypersurfaces with almost constant mean curvature, i.e. a stability estimate associated with Alexandrov's \cite{Alexandrov} symmetry result for constant mean curvature hypersurfaces in $\real^N$.  In fact the connection between these two symmetry results is more than just an analogy, following the ideas of \cite{MR0474149}, Ros \cite{MR996826} was able to use Serrin's symmetry result to prove Alexandrov's theorem.  The stability problem for Alexandrov's theorem has been recently studied by Ciraolo and Maggi \cite{Ciraolo:2017aa}, and also by Krummel and Maggi \cite{MR3627438} where one of their results is an $L^2$ type bound similar to ours, under the assumption that $\partial \Omega$ can be written as a $C^{1,1}$ graph over the sphere.   In fact our result, since it gives control of the asymmetry in terms of the $L^2(\partial \Omega)$ norm of $|Du|^2-1$, could be used to prove a similar symmetry result for almost constant mean curvature hypersurfaces, this is already explained in quite a bit of detail in \cite{Ciraolo:2017aa} (see the introduction and Lemma $2.3$).  

We were made aware, after we completed this work, of a paper by Magnanini and Poggesi \cite{Magnanini:2016aa} studying the Alexandrov problem which used several similar ideas to our paper. Basically, our Proposition~\ref{prop: fund est} is close to their Theorem~$2.1$, but the methods need to diverge after that point since the estimate for Serrin's problem is in a weaker norm.

We mention one last connection, with the Faber-Krahn inequality which is typically stated for the first Dirichlet eigenvalue but also has a version for the torsional rigidity $\lambda(\Omega)$ (defined in \eref{ulambdadef}). In this case the Faber-Krahn inequality says,
\begin{equation}\label{e.FK}
\lambda(\Omega) \geq \lambda(B) \ \hbox{ for the ball $B$ with } \ |B| = |\Omega|.
\end{equation}
This follows from the fact that the Dirichlet energy is non-increasing with respect to Schwarz symmetrization.  We can derive immediately from the Faber-Krahn inequality that any minimizer of $\mathcal{J}$ over the class of open sets of $\real^N$ must be radially symmetric, then direct computation yields that $B_{r_*}$ is the only energy minimizer.  Thus the Faber-Krahn inequality, and its related stability results, have an important connection with the energy $\mathcal{J}$ and its gradient flows.  In a recent paper Brasco, De Philippis and Velichkov \cite{Brasco:2015aa} have proven an optimal stability result for the Faber-Krahn (as well as a range of related inequalities).  Basically, in our current understanding, the relationship between our result and the stability of the Faber-Krahn inequality is analogous to the relationship between the stability of almost constant mean curvature hypersurfaces \cite{Ciraolo:2017aa} and the stability of the isoperimetric inequality (Figalli, Maggi and Pratelli \cite{Figalli:2010aa}).  

The Faber-Krahn and Isoperimetric stability results are proving a lower bound 
\[E(\Omega) - E(B) \gtrsim d(\Omega,B)^2,\]
 where $E$ is the associated energy functional, Dirichlet or perimeter, $B$ is the energy minimizer and $d$ is an appropriate distance.  Results of the type considered in our paper or in \cite{Ciraolo:2017aa} are proving a lower bound of 
 \[ g_\Omega(\grad E(\Omega),\grad E(\Omega)) \gtrsim d(\Omega,B)^2\]
  where, formally, $g$ is a metric on a manifold of subsets of $\real^N$ associated to the distance $d$.  For our problem,
  \[ g_\Omega(f,g) = \int_{\partial \Omega} f g \ d \sigma.\]
  This interpretation suggests why this type of gradient stability estimate may require more regularity than energy stability estimates, and also explains the fundamental connection with gradient flows.  For smooth functions on finite dimensional spaces both of these estimates follow from a lower bound of $\grad ^2 E$, but as of yet we do not see how to manifest such a connection in our setting.

Now we make explicit the assumptions on the domain $\Omega \subset \real^N$ which we use for our result.  
\begin{enumerate}[{(a)}]
\item $\Omega$ is connected and has $C^2$ boundary. \label{a}\vspace{.5mm}
\item $\Omega$ has an interior ball of radius $\rho_0>0$ at every boundary point. \label{b}\vspace{.5mm}
\item $\Omega$ is a $L_0$-John domain, i.e. there is a base point $x_0 \in \Omega$ such that each point $x \in \Omega$ can be joined to $x_0$ by a curve $\gamma : [0,1] \to \Omega$ such that, \label{c}
\[ d(\gamma(t),\partial \Omega) \geq L_0^{-1} |\gamma(t)-x|.\]
\end{enumerate}
Assumption~\eqref{a} is qualitative, assumption~\eqref{b} implies quantitative non-degeneracy of $u$ near $\partial \Omega$ and assumption \eqref{c} is a quantification of connectivity and it is exactly what is needed to prove Poincar\'{e}-type inequalities in $\Omega$.  Instead of~\eqref{c} we could assume that $\Omega$ is a Lipschitz domain.  One thing we would like to emphasize is that, although we do require a certain amount of regularity, these assumptions do not place us in a perturbative regime.

 To simplify notation we will just take $r_0 = \rho_0$.
\begin{theorem}\label{thm: main stability}
Suppose that $\Omega$ satisfies Assumptions~\eqref{a}, \eqref{b}, and \eqref{c}. Then for some ball $B_{r_*}$ of radius $r_*$,
\begin{equation}\label{e.opt}
 \frac{|\Omega \Delta B_{r_*}|}{|B_{r_*}|} \leq C\left(N,L_0,\frac{\textup{diam}(\Omega)}{\rho_0},\frac{\textup{diam}(\Omega)}{r_*}\right)\left(\frac{1}{r_*^{N-1}} \int_{\partial \Omega} (|Du|^2 - 1)^2 d\sigma\right)^{1/2}.
 \end{equation}
\end{theorem}
See the beginning of Section~\ref{sec: stability} for an outline of the proof.  Further discussion about possible modifications of the assumptions can be found throughout Section~\ref{sec: stability} in the course of the proof.

\subsection{Linear stability of equilibria for the droplet model}\label{sec: hypothesis} Next we explain the application of Theorem~\ref{thm: main stability} to the dynamic stability of the droplet problem \eref{dropletprob}.

Before we describe our result we discuss the literature on \eref{dropletprob}.  For more details and literature see \cite{FKdrops}.  The quasi-static limit leading to problems of the form \eref{dropletprob} was first studied by Greenspan~\cite{Greenspan:1976aa,greenspan}.  Grunewald and Kim~\cite{GrunewaldKim11} construct global in time weak (energy) solutions from general initial data by a discrete gradient flow scheme.  The author and Kim \cite{FKdrops} constructed global in time time viscosity solutions which converge to the equilibrium (although without a rate) under a certain approximate reflection symmetry condition on the initial data (see Section~\ref{sec: exponential rate} where we recall the precise condition).  See also Glasner \cite{MR2144627,MR2221703,glasner}, Glasner, Kim \cite{GlasnerKim09} and Mellet \cite{MR3268920}.  A different approach was taken by Escher and Guidotti~\cite{Escher:2015aa} who show local existence of smooth solutions by writing the equation in a coordinate system adapted to the initial data.  Most relevant to our result, although by a completely different method, is the work of Guidotti~\cite{Guidotti:2015aa} showing the stability of the equilibrium state for \eref{dropletprob} by a perturbative approach.  

We also show the dynamic stability of the equilibrium state for \eref{dropletprob} with exponential convergence, but our method, which was outlined above, is completely different from \cite{Guidotti:2015aa}.  We use the energy dissipation estimate \eref{dissipation} in combination with our new stability result Theorem~\ref{thm: main stability} to obtain the exponential rate.  As far as we are aware this is one of the first applications of a stability estimate for a geometric minimization problem to show dynamic stability of an associated gradient flow.  The main benefits of our approach are $(1)$ it does not require to be in the perturbative regime, $(2)$ it has an appealing connection with the gradient flow structure of the problem, $(3)$ given the stability estimate Theorem~\ref{thm: main stability} and regularity theory for the flow the energy decay is just an elementary Gr\"{o}nwall argument.  We state our result below.

\begin{theorem}\label{thm: dynamic stability}
Suppose that the solution $\Omega_t$ of the droplet problem \eref{dropletprob} satisfies assumptions~\eqref{a}, \eqref{b}, and \eqref{c} uniformly for all $t>0$ and $\textup{diam}(\Omega)$ remains bounded. Then there are constants $C,c>0$ depending on $N$ and the suprema in $t$ of $L_0,\frac{\textup{diam}(\Omega_t)}{\rho_0}$, and $\frac{\textup{diam}(\Omega_t)}{r_*}$ such that,
\[  \inf_{x \in \real^N} |\Omega_t \Delta B_{r_*}(x)| \leq  C(\mathcal{J}(\Omega_0) - \mathcal{J}(B_{r_*}))e^{-cr_*^{-1}t}.\]   
\end{theorem}
Since the convergence result is conditional on the propagation of assumptions~\eqref{a}, \eqref{b}, and \eqref{c} we must explain when this is expected to be true. This will be explained in greater detail in Section~\ref{sec: conditions} but we give a brief summary here.   The result of the author and Kim~\cite{FKdrops} implies that assumption \eqref{c} and the diameter bound will be propagated for initial data satisfying a certain strong star-shapedness type condition.  The regularity assumptions~\eqref{a} and \eqref{b} have not been studied yet for this problem, but initial data with small local Lipschitz constant is expected to regularize immediately, see Choi, Jerison and Kim \cite{Choi:2009aa,CJK} for a regularity result for a similar quasi-static problem.

\subsection{Notation} We will use $C,c>0$ to denote dimension dependent constants which may change from line to line.  If the constant has additional dependencies on some parameters $A_1,A_2,\dots$ we will write $C(A_1,A_2,\dots)$.  The only exception to this rule should be in heuristic remarks (outside of proofs) explaining ideas where we will not include all the dependencies of the constants.
\subsection{Acknowledgments}
I would like to thank Inwon Kim, Francesco Maggi, and John Garnett for inspiring discussions and helpful comments.

\section{Stability result for regular boundary}\label{sec: stability}
We explain the strategy to obtain the stability estimate of Theorem~\ref{thm: main stability}.  The starting point is the ideas of \cite{BNST08,BNST-alt,BNST09,BNST0} where it was discovered that the symmetry property of Serrin's problem is closely related to a certain arithmetic-geometric mean inequality.  We are able to improve on their calculation to obtain the following fundamental estimate,
\begin{equation*}
 \int_\Omega u \left( \left(\tfrac{\Delta u}{N}\right)^2-  \tfrac{1}{{ N \choose 2 }}S_2(D^2u)\right) dx \leq C_N\int_{\partial\Omega}\left| \tfrac{\lambda(\Omega)}{N}(x-x_0) +Du\right| \left||Du|^2-1\right| \ d \sigma. 
 \end{equation*}
 where $S_2$ is the second symmetric function of the eigenvalues and is defined below in Section~\ref{sec: symmetric}.  Via the AM-GM inequality this yields,
 \begin{equation}\label{e.fundest0}
 \int_\Omega u(x) |D^2u(x)+\tfrac{\lambda(\Omega)}{N}\Id|^2 dx \leq C_N\int_{\partial\Omega}\left| \tfrac{\lambda(\Omega)}{N}(x-x_0) +Du\right| \left||Du|^2-1\right| \ d \sigma.
 \end{equation}
 We will need to exploit this weighted Sobolev norm estimate to obtain symmetry.  We explain the idea of how to use \eref{fundest0}.  Let $x_0 \in \Omega(u)$ be a point where $\max u$ is obtained.  Define,
\begin{equation}\label{e.wdef1}
 w(x) = u(x) - (u(x_0) - \tfrac{\lambda(\Omega)}{2N}|x-x_0|^2) \ \hbox{ with } \ w(x_0) = 0, \ Dw(x_0) = 0,
 \end{equation}
and $w$ is harmonic.  Then we derive from \eref{fundest0},
\[  \int_\Omega u(x) |D^2w|^2 dx \leq \|Dw\|_{L^2(\partial \Omega)}\||Du|^2-1\|_{L^2(\partial \Omega)}.\]
Now one is tempted to use a trace theorem for the $L^2$ norm, weighted by $u(x)$, to bound,
\begin{equation}\label{e.trace thm 0}
\|Dw\|_{L^2(\partial \Omega)} \leq C \left( \int_\Omega u(x) |D^2w|^2 dx \right)^{1/2}.
\end{equation}
When $u$ is comparable to the distance function to $\real^N \setminus \Omega$ this estimate has the correct scaling, so perhaps it seems reasonable to expect.  Now in reality the estimate \eref{trace thm 0} is false in general, it is an endpoint case which fails for some logarithmically singular functions, however a very similar estimate \emph{does} hold for harmonic functions (like $w,Dw$) via a simple integration by parts identity, see Section~\ref{sec: trace thm}. Anyway, pretending we can use \eref{trace thm 0}, we would obtain,
\[ \int_\Omega u(x) |D^2w|^2 dx \leq C \int_{\partial \Omega} (|Du|^2-1)^2 \  d \sigma \ \hbox{ and } \ \int_{\partial \Omega} |Dw|^2 dx \leq C  \int_{\partial \Omega} (|Du|^2-1)^2 \ d \sigma.\]
From the second estimate we can easily derive,
\[\int_{\partial \Omega} (\tfrac{\lambda(\Omega)}{N}|x-x_0| - 1)^2 dx \leq C  \int_{\partial \Omega} (|Du|^2-1)^2 \ d \sigma,\]
which is now obviously a type of distance estimate to the ball $B_{\lambda(\Omega)/N}(x_0)$.  With some more work we can obtain the estimate in measure.  This is the basic outline of the proof.  
\subsection{Convexity inequalities and $k$-Hessian equations}\label{sec: symmetric} Just as the AM-GM inequality underlies the isoperimetric inequality and it's corresponding stability estimates, a closely related convexity inequality underlies the stability of Serrin's problem.  Let $M$ be an $N \times N$ symmetric matrix with real entries, and call the $N$ real eigenvalues of $M$, $\mu_1,\dots, \mu_N$.  The $k$-th symmetric function of the eigenvalues of $M$ is,
$$ S_k(M) = \tfrac{1}{{N \choose k }}\sum_{ i_1 < \cdots < i_k} \mu_{i_1}\cdots\mu_{i_k}.$$
When $k=1$ this is the trace, and when $k=N$ it is the determinant.  There is a classical refinement of the arithmetic geometric mean inequality which gives,
\begin{equation}\label{eqn: general amgm}
\det(M)^{1/N} = S_N(M)^{1/N} \leq \cdots \leq S_{2}(M)^{1/2} \leq S_1(M) = \tfrac{1}{N}\Tr(M),
\end{equation}
with strict inequality holding unless $\mu_1 = \cdots =\mu_N$.  

As was discovered by \cite{BNST08,BNST09,BNST-alt}, the symmetry of solutions of \eref{serrin} is fundamentally related to the AM-GM inequality between $\Tr$ and $S_2$.  We note that this inequality can be further quantified as:
\begin{lemma} \label{lem: quadratic growth}
There is a dimensional constant $c_N>0$ so that,
\[ S_1(M)^2-S_2(M)\geq c_N|M-S_1(M)\Id|^2.\]
\end{lemma}
\noindent This can be checked by direct computation.

We state several more useful facts about $S_2(M)$, proofs of these identities can be found in Reilly \cite{reilly1974}, see also Wang \cite{Wang:2009aa} where these $k$-Hessian equations are studied.  If $M$ has entries $m_{ij}$ then we call,
\[ S_2^{ij}(M) = \frac{\partial S_2(M)}{\partial m_{ij}}, \]
and since $S_2(M)$ is homogeneous of degree $2$ on $\real^{2N}$ we have the identity,
\[ S_2(M) = \tfrac{1}{2}S_2^{ij}(M)m_{ij}. \]
Now $S_2^{ij}$ is divergence free,
\[ D_{i}S_2^{ij}(D^2u) = 0,\]
and so $S_2(D^2u)$ can be written in divergence form,
\[ S_2(D^2u) = \tfrac{1}{2}D_i(S^{ij}_2(D^2u)D_ju).\]
Finally we have the following identity relating $S_2(D^2u)$ with the curvature of the level sets of $u$,
\begin{equation}
 \frac{S_2^{ij}(D^2u)D_iuD_ju}{|Du|^2} = \kappa |Du| \ \hbox{ where } \ \kappa = |Du|^{-1} \Tr((Id-\widehat{Du}\otimes\widehat{Du})D^2u). 
 \end{equation}

\subsection{Integration by parts identities} Before proceeding with the proofs we establish several integration by parts identities which will be useful later.  This first identity was also used in an important way in \cite{BNST08},
\begin{align*}
\int_{\Omega} |Du|^2 dx &= \int_\Omega |Du|^2\frac{-\Delta u}{\lambda(\Omega)} \ dx \\
&= \tfrac{2}{\lambda(\Omega)}\int_\Omega \langle D^2u Du,Du\rangle dx + \tfrac{1}{\lambda(\Omega)}\int_{\partial\Omega} |Du|^3 d\sigma(x) \\
& = \tfrac{2}{\lambda(\Omega)}\int_\Omega |Du|^2\Delta u - \Tr((Id-\widehat{Du}\otimes\widehat{Du})D^2u)|Du|^2 \ dx + \tfrac{1}{\lambda(\Omega)}\int_{\partial\Omega} |Du|^3 d\sigma(x) .
\end{align*}
Recalling from above that $\kappa = |Du|^{-1} \Tr((Id-\widehat{Du}\otimes\widehat{Du})D^2u)$ is $N-1$ times the mean curvature of the level sets of $u$ and rearranging above,
\begin{equation}\label{e.identity1}
 \int_\Omega \kappa |Du|^3 dx= -\tfrac{3\lambda(\Omega)^2}{2}\textup{Vol}+\tfrac{1}{2} \int_{\partial \Omega} |Du|^3 d\sigma(x). 
 \end{equation}
It turns out that the term $\int_{\partial \Omega} |Du|^3 d\sigma(x)$ in the above identity is not so convenient to work with by itself, we establish the following new identity,
\begin{lemma}\label{lem: cube}
For any $x_0 \in \real^N$,
\begin{equation}\label{e.cube}
 \int_{\partial\Omega} |Du|^3 d \sigma = \tfrac{\lambda(\Omega)^2(N+2)}{N}\textup{Vol} - \int_{\partial\Omega}\langle \tfrac{\lambda(\Omega)}{N}(x-x_0) +Du,\nu \rangle (|Du|^2-1) \ d \sigma.
 \end{equation}
\end{lemma}  
Using this identity in \eref{identity1} gives,
\begin{equation}\label{e.identity2}
\int_\Omega \kappa |Du|^3 dx= -\tfrac{N-1}{N}\lambda(\Omega)^2\textup{Vol}-\tfrac{1}{2}  \int_{\partial\Omega}\langle \tfrac{\lambda(\Omega)}{N}(x-x_0) +Du,\nu \rangle (|Du|^2-1) \ d \sigma
\end{equation}
\begin{proof}[Proof of Lemma~\ref{lem: cube}]
We will need a Pohozaev type identity, for any $x_0 \in \real^N$,
\begin{align*}
 -\int_\Omega \langle x-x_0, D u \rangle \Delta u \ dx&=  \int_\Omega |D u|^2+\langle x-x_0,D^2u Du\rangle \ dx + \int_{\partial \Omega} \langle x-x_0, D u \rangle |Du| d \sigma \\
 & = \lambda(\Omega)\textup{Vol}+\int_\Omega \langle x-x_0,\tfrac{1}{2}D(|Du|^2)\rangle \ dx-\int_{\partial \Omega} \langle x-x_0, \nu \rangle |Du|^2 d \sigma \\
 &= -\frac{N-2}{2}\lambda(\Omega)\textup{Vol} -\tfrac{1}{2}\int_{\partial \Omega} \langle x-x_0, \nu \rangle |Du|^2 d \sigma
 \end{align*}
 and on the other hand by using the equation $-\Delta u = \lambda(\Omega)$ and then integrating by parts,
 \[ -\int_\Omega \langle x-x_0, D u \rangle \Delta u \ dx = -N\lambda(\Omega)\textup{Vol}.\]
 Combining these identities we find,
 \[\int_{\partial \Omega} \langle \tfrac{\lambda(\Omega)}{N}(x-x_0), \nu \rangle |Du|^2 d \sigma= \frac{N+2}{2N}\lambda(\Omega)^2\textup{Vol}.\]
Now we can compute the desired identity,
\begin{align*}
  \int_{\partial\Omega} |Du|^3 d \sigma &= -\int_{\partial\Omega} |Du|^2\langle Du,\nu \rangle \ d \sigma \\
  &= -\int_{\partial\Omega} (|Du|^2-1)\langle Du,\nu \rangle \ d \sigma - \int_{\partial\Omega} \langle Du,\nu \rangle \ d \sigma  \\
  &=\int_{\partial\Omega}\langle \tfrac{\lambda(\Omega)}{N}(x-x_0),\nu \rangle |Du|^2 \ d\sigma-\int_{\partial\Omega} (|Du|^2-1)\langle Du+\tfrac{\lambda(\Omega)}{N}(x-x_0),\nu \rangle \ d \sigma \\
  &\cdots+  \int_{\Omega} (-\Delta u) \ d \sigma -\int_{\partial \Omega}\langle \tfrac{\lambda(\Omega)}{N}(x-x_0),\nu \rangle \ d \sigma\\
  &=\tfrac{\lambda(\Omega)^2(N+2)}{N}\textup{Vol}-\int_{\partial\Omega} (|Du|^2-1)\langle Du+\tfrac{\lambda(\Omega)}{N}(x-x_0),\nu \rangle \ d \sigma
  \end{align*}
  where the last two terms were equal using $-\Delta u = \lambda(\Omega)$ in $\Omega$ and divergence theorem to evaluate $\int_{\partial \Omega}\langle \tfrac{\lambda(\Omega)}{N}(x-x_0),\nu \rangle \ d \sigma = \int_{\Omega} \lambda(\Omega) \ dx$.
\end{proof}

\subsection{A weighted $L^2$ estimate on the Hessian}\label{sec: hessian est} In this section we establish the fundamental estimate which, with sufficient boundary regularity, we will be able to exploit to obtain the stability result.  We state the result as a Proposition,
\begin{proposition}\label{prop: fund est}
For any $x_0 \in \real^N$,
\[ \int_\Omega u \left( \left(\tfrac{\Delta u}{N}\right)^2-  \tfrac{1}{{ N \choose 2 }}S_2(D^2u)\right) dx \leq C_N\int_{\partial\Omega}\left| \tfrac{\lambda(\Omega)}{N}(x-x_0) +Du\right| \left||Du|^2-1\right| \ d \sigma. \]
\end{proposition}
Using Lemma~\ref{lem: quadratic growth} we also have the estimate,
\begin{equation}\label{e.fund est}
 \int_\Omega u(x) |D^2u(x)+\tfrac{\lambda(\Omega)}{N}\Id|^2 dx \leq C_N\int_{\partial\Omega}\left| \tfrac{\lambda(\Omega)}{N}(x-x_0) +Du\right| \left||Du|^2-1\right| \ d \sigma
 \end{equation} 

The main ideas of the computation have already been established in \cite{BNST08, BNST-alt}, however we have made an important improvement.  In the result of \cite{BNST08} (see \cite{BNST08} Lemma $3.1$) the analogous estimates had mismatched homogeneity, quadratic in $D^2u(x)+\tfrac{\lambda(\Omega)}{N}\Id$ and linear in $|Du|-1$.  Our estimate \eref{fund est} is quadratic on both the left hand side and the right hand side.

\begin{proof}[Proof of Proposition~\ref{prop: fund est}]
The beginning of the computation follows the idea of \cite{BNST08},
\begin{align*}
\textup{Vol}  = \int_\Omega \left( \lambda(\Omega)^{-1}\Delta u\right)^2 u  \ dx 
& \geq \tfrac{N^2}{\lambda(\Omega)^2}\int_\Omega  \frac{1}{{ N \choose 2 }}S_2(D^2u) u  \ dx \\
& = \tfrac{N}{(N-1)\lambda(\Omega)^2}\int_\Omega uS_2^{ij}(D^2u)D_{ij}^2u  \ dx \\
& = \tfrac{N}{(N-1)\lambda(\Omega)^2}\int_\Omega uD_i(S_2^{ij}(D^2u)D_{j}u) dx \\
& = -\tfrac{N}{(N-1)\lambda(\Omega)^2}\int_\Omega S_2^{ij}(D^2u)D_{i}uD_ju \ dx \\
& = -\tfrac{N}{(N-1)\lambda(\Omega)^2}\int_\Omega \kappa |Du|^3 \ dx  \\
& = \tfrac{3N}{2(N-1)}\textup{Vol} - \tfrac{N}{2(N-1)\lambda(\Omega)^2} \int_{\partial\Omega} |Du|^3 d \sigma(x).
\end{align*}
Here is where our computation diverges from that of \cite{BNST08}.  We insist on achieving an error estimate which is quadratic in $|Du|-1$ to match the quadratic homogeneity of $[{ N \choose 2 }^{-1}S_2(D^2u) - N^{-1}\Delta u]$.  Instead of adding and subtracting $1$ in last boundary integral above, which leads to a linear order error term $|Du|^3 - 1$, we use the identity \eref{cube} and find a quadratic term.  We finish the computation using \eref{cube},
\begin{align*}
\textup{Vol}  = \int_\Omega \left( \lambda(\Omega)^{-1}\Delta u\right)^2 u  \ dx 
& \geq \tfrac{N^2}{\lambda(\Omega)^2}\int_\Omega  \frac{1}{{ N \choose 2 }}S_2(D^2u) u  \ dx \\
& = \tfrac{3N}{2(N-1)}\textup{Vol} - \tfrac{N}{2(N-1)\lambda(\Omega)^2} \int_{\partial\Omega} |Du|^3 d \sigma(x) \\
&= \textup{Vol} + \tfrac{N}{2(N-1)\lambda(\Omega)^2} \int_{\partial\Omega}\langle \tfrac{\lambda(\Omega)}{N}(x-x_0) +Du,\nu \rangle (|Du|^2-1) \ d \sigma,
\end{align*}
where $x_0$ is an arbitrary point in $\real^N$.  Rearranging the above computation, what we have finally established is,
\[ \int_\Omega u \left( \left(\tfrac{\Delta u}{N}\right)^2-  \tfrac{1}{{ N \choose 2 }}S_2(D^2u)\right) dx \leq \tfrac{-1}{2N(N-1)}\int_{\partial\Omega}\langle \tfrac{\lambda(\Omega)}{N}(x-x_0) +Du,\nu \rangle (|Du|^2-1) \ d \sigma. \]
Finally this was the desired estimate of Proposition~\ref{prop: fund est}.
\end{proof}

\subsection{A weighted norm trace theorem for harmonic functions}\label{sec: trace thm}  In this section we make rigorous the trace theorem claimed in \eref{trace thm 0}.  For a given measurable weight function $\omega\geq 0$ we will write $L^2_\omega$ for the weighted $L^2$-space with norm,
\[ \|f\|_{L^2_w} = \int |f|^2 \omega(x) \ dx.\]
We are interested to prove the following type of estimate,
\[ \int_{\partial \Omega} |f-(f)_\Omega|^2 \ d\sigma \leq C \int_{\Omega} |Df|^2 u(x) \ dx.\]
To simplify matters we start by replacing the weight $u(x)$ by the distance function to $\real^N \setminus \Omega$ which we call $d_\Omega(x)$.  Now we consider the possible validity of the following trace estimate
\[ \int_{\partial \Omega} |f - (f)_{\partial \Omega}|^2 \ dx   \leq C \int_\Omega |Df|^2 d_\Omega(x) \ dx.\]
This estimate is an endpoint case and it is \emph{false} for general smooth $f$.  Take as an example $|\log d_\Omega(x)|^{\alpha}$ for any $0 < \alpha <1/2$, the gradient has finite $L^2_{d_\Omega}(\Omega)$ norm, but obviously there cannot be a trace on $\partial \Omega$.

\medskip

In order to obtain a valid estimate we will need to use that the particular $f$ we will be working with satisfies a stronger property, it is harmonic in $\Omega$.  We make use of the following identity.  Let $f$ be a harmonic function in $\Omega$ then,
\begin{align*}
 \int_{\Omega} |Df|^2u(x) \ dx &= -\int_{\Omega} f D f \cdot D u \ dx \\
 &= \int_{\Omega} f^2 \Delta u + f D f \cdot D u \ dx - \int_{\partial \Omega} f^2 Du \cdot \nu \ d\sigma.
 \end{align*}
 Using the middle equality we find,
 \[ -\int_{\Omega} f D f \cdot D u \ dx = \frac{1}{2} \int_{\Omega} f^2 \Delta u \ dx+\frac{1}{2}\int_{\partial \Omega} f^2  |Du|\ d\sigma.\]
 This can be rearranged to,
 \begin{equation}\label{e.traceid}
  \int_{\partial \Omega} f^2 |Du| \ d\sigma = 2\int_{\Omega} |Df|^2u(x) \ dx +\lambda(\Omega)\int_{\Omega} f^2  \ dx,
  \end{equation}
holding for every $f$ harmonic in $\Omega$ and continuous up to $\partial \Omega$.  

\medskip

Now we can state the main result of this section which, in particular, will apply to (a slight modification of) the function $Dw$ defined in \eref{wdef1} in Section~\ref{sec: hessian est}.

\begin{proposition}\label{prop: trace}
Suppose that $\Omega$ is an $L_0$-John domain with base point $x_0$ and $f$ is harmonic in $\Omega$, continuous up to $\partial \Omega$, with $f(x_0) = 0$.  Then we have the following estimates,
\[ \int_{\partial \Omega} f^2 |Du| \ d\sigma \leq 2 \int_{\Omega} |Df|^2u(x) \ dx + CL_0^N\lambda(\Omega)|\Omega|^{1/N}\int_{\Omega} |Df|^2d_\Omega(x) \ dx\]
\end{proposition}

\medskip

Now all that we need to go from the trace identity \eref{traceid} to Proposition~\ref{prop: trace} is to control $\|f\|_{L^2(\Omega)}$ by the weighted $L^2_{u}$ norm of $Df$.  We will use a Poincar\'{e} inequality with weighted norm proved by Hurri-Syrj\"{a}nen \cite{Hurri-Syrjanen:1994aa}.  This is the only place where we use the assumption that $\Omega$ is a John domain, as it is exactly suited to this type of weighted Poicar\'{e} inequality.  We have a small wrinkle which is that $f$ does not have mean $0$ on $\Omega$, instead we need to use that $f(x_0) = 0$ and that $f$ is harmonic.  This will require us to state the result of \cite{Hurri-Syrjanen:1994aa} carefully.

The following result is not exactly what is stated in \cite{Hurri-Syrjanen:1994aa} but an inspection of the proof will see that this result is indeed proven there.
\begin{theorem}[Hurri-Syrj\"{a}nen]\label{lem: dist weighted poincare}
Let $\Omega$ a $L_0$-John domain in $\real^n$ with base point $x_0$ and $d_\Omega(x)$ be the distance function to $\real^N \setminus \Omega$.  There is a dimensional constant $C>0$ such that for any smooth $f: \Omega \to \real$,
\[ \| f - (f)_{Q_0}\|_{L^{\frac{2N}{N-1}}(\Omega)}   \leq CL_0^N\|Df\|_{L^2_{d_\Omega}(\Omega)},\]
where $Q_0$ is any cube centered at $x_0$ with $\textup{diam}(Q_0) \leq d(Q_0,\partial \Omega) \leq 4 \ \textup{diam}(Q_0)$.
\end{theorem}
\noindent We remark that since we only need a bound for the $L^2(\Omega)$ norm in \eref{traceid} the assumptions, either on $\Omega$ or on the weight, can be weakened somewhat and we could still derive a version of Proposition~\ref{prop: trace}.

 Now since we have carefully stated Theorem~\ref{lem: dist weighted poincare} we can use the property that $f$ is harmonic with $f(x_0) = 0$ in place of having $(f)_\Omega = 0$.
\begin{lemma}\label{lem: poincare}
Let $\Omega$ a $L_0$-John domain in $\real^n$ with base point $x_0$ and $d_\Omega(x)$ be the distance function to $\real^N \setminus \Omega$. There is a dimensional constant $C>0$ such that if $f$ is a harmonic function in $\Omega$ with $f(x_0) = 0$ then,
\[ \| f \|_{L^{\frac{2N}{N-1}}(\Omega)}   \leq CL_0^N\|Df\|_{L^2_{d_\Omega}(\Omega)}.\]
\end{lemma}
From the Lemma and identity \eref{traceid}, Proposition~\ref{prop: trace} follows immediately.
\begin{proof}[Proof of Lemma~\ref{lem: poincare}]
  Call $B_0 = B_{\textup{diam}(Q_0)/2}(x_0)$.  Then by the assumption on $Q_0$,
\[ Q_0 \subset B_0 \subset \Omega.\]
 Since $f$ is harmonic and $f(x_0) = 0$ we have by the mean value property,
\[ (f)_{B_0}= 0.\]
The usual Poincar\'{e} inequality in $B_0$ implies then that,
\begin{align*}
 \| f \|_{L^{\frac{2N}{N-1}}(Q_0)} \leq \| f \|_{L^{\frac{2N}{N-1}}(B_0)}  &\leq C|Q_0|^{\frac{1}{N}+\frac{N-1}{2N}-\frac{1}{2}}\|Df\|_{L^2(B_0)} \\
 &\leq C|Q_0|^{\frac{1}{N}+\frac{N-1}{2N}-\frac{1}{2}-\frac{1}{2N}}\|Df\|_{L^2_{d_\Omega}(B_0)} \\
 &\leq C\|Df\|_{L^2_{d_\Omega}(\Omega)}.
  \end{align*}
  Now we apply the weighted Poincar\'{e} inequality Theorem~\ref{lem: dist weighted poincare} and the above estimate,
  \[ \| f\|_{L^{\frac{2N}{N-1}}(\Omega)} \leq \| f - (f)_{Q_0}\|_{L^{\frac{2N}{N-1}}(\Omega)} +  \left(|\Omega|/|Q_0|\right)^{\frac{N-1}{2N}}\| (f)_{Q_0} \|_{L^{\frac{2N}{N-1}}(Q_0)} \leq CL_0^N\|Df\|_{L^2_{d_\Omega}(\Omega)} \]
 which is the desired result.  Note we have used the John condition with base point $x_0$ and a point $x \in \Omega$ with $|x-x_0| \geq \textup{diam}(\Omega)/2$ to find that $|Q_0| \geq \textup{diam}(\Omega)^N/L_0^N$.
 \end{proof}

\subsection{An $L^2$ distance estimate on $\partial \Omega$} Now we are able apply the trace inequality Proposition~\ref{prop: trace} in combination with Proposition~\ref{prop: fund est} to obtain an $L^2$-type estimate on the distance of $\partial \Omega$ to $\partial B_{N/\lambda(\Omega)}$.
\begin{proposition}\label{prop: l2 dist est}
Suppose that $\Omega$ satisfies assumptions~\eqref{a}, \eqref{b} and \eqref{c}.  Then there exists $C>0$ depending on $N,L_0,\frac{\textup{diam}(\Omega)}{\rho_0}$ such that, 
\[\inf_{x_0 \in \real^N}\left( \int_{\partial \Omega} (\tfrac{\lambda(\Omega)}{N}|x-x_0| - 1)^2 \ d\sigma(x)\right)^{1/2} \leq C\left(\int_{\partial \Omega} (|Du|^2-1)^2 \ d \sigma(x)\right)^{1/2}.\]
\end{proposition}
The proof is mainly carrying out the formal argument found in Section~\ref{sec: hessian est}, we use the interior ball condition, assumption~\eqref{b}, to apply the Hopf Lemma and obtain a lower bound 
\[u(x)  \geq c_N\frac{\rho_0}{\textup{diam}(\Omega)} d_\Omega(x)\]
 and then apply the trace inequality Proposition~\ref{prop: trace}.   At this stage we can also see some indication that the interior ball assumption (or a slight weakening) on $\partial \Omega$ is necessary.  When $\partial \Omega$ has only Lipschitz regularity $|Du|^{-1}$ will no longer be in $L^{\infty}(\partial \Omega)$ and $u(x)$ will no longer have a lower bound by the distance function.

\begin{proof}
The first thing we prove is the lower bound,
\begin{equation}\label{e.ulb}
 u(x) \geq c_N\textup{diam}(\Omega)^{-1} \rho_0 d_\Omega(x). 
 \end{equation}
Let $x_0 \in \Omega$ and call $y \in \partial \Omega$ to be the point achieving $d_\Omega(x_0) = |x_0-y|$.  There is a ball $B_{\rho_0}(z)$ touching $\partial \Omega$ from the inside at $y$.  Since $\partial \Omega$ is $C^1$ we must have,
\[ \frac{y-z}{|y-z|} = \frac{y-x_0}{|y-x_0|} = \nu(y) \ \hbox{ the inward unit normal vector to $\partial \Omega$ at $y$.}\]
We construct a barrier to get a lower bound on $u(z)$,
\[ \varphi(x) = \frac{\lambda(\Omega)}{2N}(\rho_0^2 - |x-z|^2).\]
By the maximum principle $u(x) \geq \varphi(x)$ and,
\[ u(x_0) \geq \varphi(x_0) \geq c_N \lambda(\Omega)\rho_0 (\rho_0 -|x_0-z|) = c_N \lambda(\Omega)\rho_0d_\Omega(x_0).\]
Finally by the monotonicity of $\lambda(\Omega)$,
\begin{equation}\label{e.llb}
 \lambda(\Omega) \geq \lambda(B_{\textup{diam}(\Omega)/2}) = c_N \textup{diam}(\Omega)^{-1},
 \end{equation}
and now we have \eref{ulb}.

Now we make rigorous the heuristic argument described at the beginning of the section.  Let $x_0$ be the base point from the John domain property, assumption~\eqref{c}.  Now define,
 \[ v(x) = a - \frac{\lambda(\Omega)}{2N}|x-x_1|^2 \ \hbox{ and } \ w = u - v\]
where 
\begin{equation}\label{e.pchoice}
 x_1 = x_0 - \tfrac{N}{\lambda(\Omega)}Du(x_0) \ \hbox{ and } \ a = u(x_0) + \tfrac{N}{2\lambda(\Omega)}|Du(x_0)|^2
 \end{equation}
 are chosen so that $w(x_0) = |Dw(x_0)| = 0$. Without loss we can assume $x_1 = 0$. Now $w$ as defined is harmonic in $\Omega$ and satisfies
\[ w(x_0) = |Dw(x_0)| = 0.\]
Now we can apply the weighted trace theorem Proposition~\ref{prop: trace} and \eref{ulb} to obtain,
\begin{align*}
 \int_{\partial \Omega} |Dw|^2 \ d\sigma &\leq C\frac{\textup{diam}(\Omega)}{\rho_0}\int_{\partial \Omega} |Dw|^2|Du| \ d\sigma \\
 &\leq C\frac{\textup{diam}(\Omega)}{\rho_0}\int_{\Omega} |D^2w|^2u(x) \ dx +CL_0^N\frac{\textup{diam}(\Omega)}{\rho_0}\int_{\Omega} |D^2w|^2d_\Omega(x) \ dx \\
 &\leq C \left(L_0,\frac{\textup{diam}(\Omega)}{\rho_0}\right)\int_{\Omega} |D^2w|^2u(x) \ dx.
 \end{align*}
Next we apply Proposition~\ref{prop: fund est} to find, dropping the dependencies of $C$,
\[ \int_{\partial \Omega} |Dw|^2 \ d\sigma \leq C  \int_{\partial \Omega} \left| \tfrac{\lambda(\Omega)}{N}(x-x_0) +Du\right| \left||Du|^2-1\right| \ d\sigma\]
Noting that $\frac{\lambda(\Omega)}{N}(x-x_0) +Du = Dw$ and applying Cauchy-Schwarz we get,
\[ \left(\int_{\partial \Omega} |Dw|^2 \ d\sigma\right)^{1/2} \leq C \left(\int_{\partial \Omega} (|Du|^2-1)^2 \ d\sigma\right)^{1/2}.\]
Now we are almost finished, we just need a triangle inequality,
\begin{align*}
 \|Dw\|_{L^2(\partial \Omega)} &\geq \left\|\tfrac{\lambda(\Omega)}{N}|x-x_0| - |Du|\right\|_{L^2(\partial \Omega)}  \\
 &\geq \left\|\tfrac{\lambda(\Omega)}{N}|x-x_0| - 1\right\|_{L^2(\partial \Omega)} - \big\||Du|-1\big\|_{L^2(\partial \Omega)},
 \end{align*}
 then we finish using $||Du| - 1| \leq ||Du|^2-1|$ since that inequality is true for all positive reals.
\end{proof}

\subsection{The distance estimate in measure}  
Now we are able to prove the stability estimate in measure, the main result of Theorem~\ref{thm: main stability}, which we restate here:
\begin{theorem}\label{thm: main stability 1}
If $\Omega$ satisfies assumptions~\eqref{a}, \eqref{b} and \eqref{c} then,
 \[ \inf_{x_0 \in \real^n}\frac{|\Omega \Delta B_{r_*}(x_0)| }{|B_{r_*}|} \leq C\left(L_0,\frac{\textup{diam}(\Omega)}{\rho_0},\frac{\textup{diam}(\Omega)}{r_*}\right)\left(\frac{1}{r_*^{N-1}}\int_{\partial \Omega}(|Du|^2-1)^2 \ d \sigma \right)^{1/2}.\]
\end{theorem}
In $N=2$ the result can be improved to get an estimate in Hausdorff distance, control of $\|D^2w\|_{L^2_{u}(\Omega)}$ will give control of $\|Dw\|_{L^4(\Omega)}$ by the Poincar\'{e} inequality Lemma~\ref{lem: poincare} which then gives control of $\|w\|_{L^{\infty}(\Omega)}$.

The main additional estimate to go from Proposition~\ref{prop: l2 dist est} to Theorem~\ref{thm: main stability} is the following bound between the $L^2$ pseudo-distance on $\partial \Omega$ and the distance in measure.

\begin{lemma}\label{lem: distance est}
Suppose that $E$ is a set with Lipschitz boundary and $r>0$ such that $|E| \leq K |B_r|$ and the in-radius has $r_{in}(E) \geq r/K$. Then it holds,
\[ \frac{|E \Delta B_r|}{|B_r|} \leq C(N,K) \left(\frac{1}{r^{N-1}}\int_{\partial E} \left(\frac{|x|}{r}-1\right)^2 d\sigma(x)\right)^{1/2}.\]
\end{lemma}
\noindent We remark that the assumption that $E$ has Lipschitz boundary is not really necessary, the same result, appropriately stated, would hold for sets of finite perimeter.

We will return to the proof of Lemma~\ref{lem: distance est} below.  First we complete the proof of Theorem~\ref{thm: main stability}.

\begin{proof}[Proof of Theorem~\ref{thm: main stability}]
We call as before $x_0$ to be the base point from the John domain property and,
 \[ v(x) = a - \frac{\lambda(\Omega)}{2N}|x-x_1|^2 \ \hbox{ and } \ w = u - v\]
where, as in \eref{pchoice}, the parameters $a,x_1$ are chosen so that $w(x_0) = Dw(x_0) = 0$. Without loss we can assume $x_1 = 0$. Recall, as \eref{llb} and using $B_{\textup{diam}(\Omega)/2L_0}(x_0) \subset \Omega$ which follows from \eqref{a},
\[ c_N \textup{diam}(\Omega)^{-1}\leq \lambda(\Omega) \leq c_NL_0 \textup{diam}(\Omega)^{-1}.\] 
  Then we get the following bounds,
\[ B_{c_N N/\lambda(\Omega)}(x_0) \subset \Omega \ \hbox{ and } \ |\Omega| \leq C(N,L_0)|B_{N/\lambda(\Omega)}|. \] 
To shorten the notation we give the name $\mu_0$ to the non-dimensional ratio $\textup{diam}(\Omega)/\rho_0$.

Now we can apply Lemma~\ref{lem: distance est} and Proposition~\ref{prop: l2 dist est} to find 
\begin{align}
\frac{|\Omega \Delta B_{N/\lambda(\Omega)}|}{|B_{N/\lambda(\Omega)}|} &\leq C(L_0) \left(\lambda(\Omega)^{N-1}\int_{\partial \Omega} \left(\tfrac{\lambda(\Omega)}{N}|x|-1\right)^2 d\sigma(x)\right)^{1/2} \notag\\
& \leq C(L_0,\mu_0)\left(\lambda(\Omega)^{N-1}\int_{\partial \Omega} (|Du|^2-1)^2 \ d \sigma(x)\right)^{1/2} \label{e.measureest1}.
\end{align}
Next we aim for an estimate of,
\begin{equation*}
\frac{N}{\lambda(\Omega)} - r_*.
 \end{equation*}
We still have a piece of information we have not used, which is the volume constraint, we use it in the following way,
\begin{align}
\lambda(\Omega)\textup{Vol} &= \int_{\Omega} |Du|^2 \ dx \notag \\
&= \int_{B_{N/\lambda(\Omega)}} |Dv|^2 \ dx + \int_{\Omega}|Du|^2- |Dv|^2 \ dx \label{e.dirichletsplit}  \\
& \quad \quad  \quad \cdots+ \int({\bf 1}_{\Omega \setminus B_{N/\lambda(\Omega)}} - {\bf 1}_{ B_{N/\lambda(\Omega)} \setminus \Omega})|Dv|^2 \ dx \notag
\end{align}
We look at each term individually, the first term is,
\[ \int_{B_{N/\lambda(\Omega)}} |Dv|^2 \ dx = \frac{N}{N+2}|B_{N/\lambda(\Omega)}|. \]
The remaining terms need to be estimated, for the middle term we use the Poincar\'{e}-type inequality Lemma~\ref{lem: poincare},
\begin{align*}
 \left|\int_{\Omega}|Du|^2- |Dv|^2 \ dx\right| &\leq \|Du+Dv\|_{L^{\frac{2N}{N+1}}(\Omega)}\|Dw\|_{L^{\frac{2N}{N-1}}(\Omega)}\\
 &\leq C\left(L_0,\mu_0\right)(\|Dv\|_{L^{\frac{2N}{N+1}}(\Omega)}+|\Omega|^{\frac{1}{2N}}\|Du\|_{L^{2}(\Omega)})\|D^2w\|_{L^{2}_{d_\Omega}(\Omega)} \\
 &\leq C\left(L_0,\mu_0\right)(|\Omega|^{\frac{N+1}{2N}}\lambda(\Omega)\textup{diam}(\Omega)+|\Omega|^{\frac{1}{2N}}\lambda(\Omega)^{1/2}\textup{Vol}^{1/2})\|D^2w\|_{L^{2}_{d_\Omega} }\\
 &\leq C\left(L_0,\mu_0\right)(|\Omega|^{\frac{N+1}{2N}}+r_*^{\frac{N+1}{2}})\||Du|^2-1\|_{L^{2}(\partial \Omega)}
 \end{align*}
 In the last inequality we have bounded $\lambda(\Omega) \textup{diam}(\Omega)$ by $c_N\textup{diam}(\Omega)/\rho_0 = c_N \mu_0$, similarly for $|\Omega|^{\frac{1}{2N}}\lambda(\Omega)^{1/2}$, and we have used Proposition~\ref{prop: fund est}.  For the final term of \eref{dirichletsplit} we estimate,
 \begin{align*}
 \left|\int({\bf 1}_{\Omega \setminus B_{N/\lambda(\Omega)}} - {\bf 1}_{ B_{N/\lambda(\Omega)} \setminus \Omega})|Dv|^2 \ dx\right| &\leq C(1+\lambda(\Omega)^2\textup{diam}(\Omega)^2)|\Omega \Delta B_{N/\lambda(\Omega}| \\
 &\leq C(\mu_0)|\Omega \Delta B_{N/\lambda(\Omega}|
 \end{align*}
 where we have just estimated $|Dv|$ by it's supremum on the region of integration, which depends on the maximum between $\textup{diam}(\Omega)$ and $N/\lambda(\Omega)$.
 
 Combining the estimates on each of the terms in \eref{dirichletsplit} we end up with the estimate,
 \begin{align*}
  \lambda(\Omega)\left| \left(\frac{N}{\lambda(\Omega)}\right)^N - r_*^N\right| &\leq \left| \left(\frac{N}{\lambda(\Omega)}\right)^{N+1} - \omega_N^{-1}(N+2)\textup{Vol}\right| \\
  &\leq C(L_0,\mu_0)(|\Omega|^{\frac{N+1}{2N}}+r_*^{\frac{N+1}{2}})\||Du|^2-1\|_{L^{2}(\partial \Omega)} \\
  & \quad \cdots +C(\mu_0)|\Omega \Delta B_{N/\lambda(\Omega}|  
  \end{align*}
 where for the first inequality we have used the elementary inequalities $|x^{n} - 1| \geq |x-1|$ for all $x>0$ and $n \geq1$ and,
 \[ |x^n - a^n| = a^n|(x/a)^n-1| \geq a^n|(x/a)-1| = a^{n-1}|x-a|.\]
 Now we can combine the estimate for $|\Omega \Delta B_{N/\lambda(\Omega)}|$ from \eref{measureest1} with the above estimates for $\frac{N}{\lambda(\Omega)} - r_*$,
 \begin{align*}
 |\Omega \Delta B_{r_*}| &\leq |\Omega \Delta B_{N/\lambda(\Omega)}| + |B_{N/\lambda(\Omega)} \Delta B_{r_*}| \\
 &\leq C(\mu_0)|\Omega \Delta B_{N/\lambda(\Omega}| + C(L_0,\mu_0)(|\Omega|^{\frac{N+1}{2N}}+r_*^{\frac{N+1}{2}})\||Du|^2-1\|_{L^2(\partial \Omega)} \\
 & \leq C(L_0,\mu_0)(\textup{diam}(\Omega)^{\frac{N+1}{2}}+r_*^{\frac{N+1}{2}})\||Du|^2-1\|_{L^2(\partial \Omega)}.
 \end{align*}
 Dividing on both sides by $r_*^{N}$ yields the desired estimate,
 \[ \frac{|\Omega \Delta B_{r_*}| }{|B_{r_*}|} \leq C\left(L_0,\mu_0,\frac{\textup{diam}(\Omega)}{r_*}\right)\left(\frac{1}{r_*^{N-1}}\int_{\partial \Omega}(|Du|^2-1)^2 \ d \sigma \right)^{1/2}.\]
\end{proof}

Now we return to the proof of Lemma~\ref{lem: distance est}.
\begin{proof}[Proof of Lemma~\ref{lem: distance est}]
Let $\e>0$ to be chosen later and call $A_\e$ to be the annulus $B_{(1+\e)r}\setminus B_{(1-\e)r}$. We may rewrite,
\[ |{E} \Delta B_r| = |{E} \cap A_\e |+|{E} \setminus B_{(1+\e)r}|+|B_{(1-\e)r}\setminus {E}|.\]
For the first term we can estimate easily,
\[ |{E} \cap A_\e |\leq |A_\e| = ((1+\e)^N-(1-\e)^N)|B_r| \leq C_N\e|B_r| ,\]
    as long as we choose $\e \leq 1$.  For the second term we use the co-area formula to rewrite,
\[ |{E} \setminus B_{(1+\e)r}| = \int_{(1+\e)r}^\infty \mathcal{H}^{N-1}({E} \cap \partial B_s) ds\]
Then using the divergence theorem,
\[ 0 \leq \int_{{E} \setminus B_s} \grad \cdot (\frac{x}{|x|}) \ dx = \int_{\partial {E} \setminus B_s} \frac{x}{|x|} \cdot \nu \ d \sigma(x) - \int_{{E} \cap \partial B_s} d \sigma(x),\]
and so we have,
\[\mathcal{H}^{N-1}({E} \cap \partial B_s) \leq \mathcal{H}^{N-1}(\partial {E} \setminus B_s).\]
Now on $\partial {E} \setminus B_s$ we have $1 \leq (r^{-1}s-1)^{-2}(r^{-1}|x|-1)^2$ and therefore,
\begin{align*}
 |{E} \setminus B_{(1+\e)r}| &\leq \int_{(1+\e)r}^\infty \mathcal{H}^{N-1}(\partial {E} \setminus B_s) ds  \\
 &\leq  \left(\int_{(1+\e)r}^\infty (r^{-1}s-1)^{-2}   \ ds\right) \left( \int_{\partial {E}} (r^{-1}|x|-1)^2 \  d\sigma(x)\right).
 \end{align*}
Calculating the integral above yields
\begin{equation}\notag
|{E} \setminus B_{(1+\e)r}| \leq \frac{r}{\e} \int_{\partial {E}} (r^{-1}|x|-1)_+^2 d\sigma(x).
\end{equation}
Now we choose $\e$ so that the two terms in the estimate are of the same size, we can choose,
\[ \e^2  = \frac{1}{r^{N-1}}\int_{\partial E} (r^{-1}|x|-1)_+^2 d\sigma(x).\]
If $\e \leq 1$ as chosen then combining the estimates we obtain,
\begin{equation}\label{e.outerbd}
 \frac{|E \setminus B_r|}{|B_r|} \leq C(N) \left(\frac{1}{r^{N-1}}\int_{\partial E} \left(\frac{|x|}{r}-1\right)^2 d\sigma(x)\right)^{1/2}
 \end{equation}
otherwise,
\[ \frac{|E \setminus B_r|}{|B_r|}  \leq (1+K) \leq (1+K) \e^{1/2} = (1+K)\left(\frac{1}{r^{N-1}}\int_{\partial E} \left(\frac{|x|}{r}-1\right)^2 d\sigma(x)\right)^{1/2}.\]
Either way the desired result holds.

Next we will obtain, by a similar argument,
\begin{equation}\label{e.innerbd}
 \frac{| B_r \setminus E|}{|B_r|}  \leq C(N,K)\left(\frac{1}{r^{N-1}}\int_{\partial E} \left(\frac{|x|}{r}-1\right)^2 d\sigma(x)\right)^{1/2}.
 \end{equation}
 By the assumption there exists $x_0$ with $B_{r/K}(x_0) \subset \Omega$.  If $|x_0| \geq (1-\frac{1}{2K})r$ then,
 \[ \frac{|E \setminus B_r|}{|B_r|} \geq c(N,K) \geq c(N,K)\frac{| B_r \setminus E|}{|B_r|},\]
and \eref{innerbd} follows from \eref{outerbd}.  Otherwise we can take $|x_0| \leq (1-\frac{1}{2K})r$.  Now let us take $h$ to be the harmonic function,
\begin{equation}
\left\{
\begin{array}{l}
-\Delta h = 0 \ \hbox{ in } \ B_r \setminus B_{r/4K}(x_0) \vspace{1.5mm}\\
h = r \ \hbox{ on } \ \partial B_r \ \hbox{ and } \ h = 0 \ \hbox{ on } \ \partial B_{r/4K}(x_0) 
\end{array}\right.
\end{equation}
It follows from Hopf Lemma and the star-shapedness of $B_r$ with respect to $B_{r/4K}(x_0)$ that there is a constant $c(N,K)$ such that,
\[ D h \cdot (x-x_0) \geq  c(N,K) \ \hbox{ on } \ \partial B_r \cup \partial B_{r/4K}(x_0).\]
It is easy to check that $Dh \cdot (x-x_0)$ is harmonic so actually we have,
\[ |Dh| \geq c(N,K) \ \hbox{ in } \ B_{r} \setminus B_{r/4K}(x_0).\]
A standard barrier argument and the sub-harmonicity of $|Dh|$ shows $|Dh| \leq C(N,K)$.  We use the divergence theorem, using again that $\frac{1}{4K}B_r (x_0)\subset B_r \subset  E$, for all $ 0 < s < r$,
\[  -\int_{\partial {E}\cap \{ h < s\} } Dh \cdot \nu_E \ d \sigma(x) + \int_{{E}^C \cap \partial \{ h <s\}} |Dh|d \sigma(x) = \int_{ \{ h <s \} \setminus {E}} \Delta h \ dx = 0 .\]
Thus we obtain, using the bounds on $|Dh|$,
\[ \mathcal{H}^{N-1}({E}^C \cap \{ h =s\}) \leq C(N,K)\mathcal{H}^{N-1}(\partial {E}\cap \{ h < s\}).\]
 Let $\e>0$, to be chosen, we use the co-area formula with the level set function $h$,
\begin{align*}
 |B_{(1-\e)r} \setminus E| & \leq C(N,K)\int_{B_{(1-\e)r}}  \indicator_{E^C}|Dh|dx\\
  &\leq C\int_{0}^{(1-c\e)r} \mathcal{H}^{N-1}({E}^C \cap \{ h = s\}) ds  \\
 &\leq C\int_{0}^{(1-c\e)r} \mathcal{H}^{N-1}(\partial {E}\cap \{ h < s\}) ds  \\
 &\leq C \left(\int_{0}^{(1-c\e)r}  (r^{-1}s-1)^{-2} \ ds\right) \left( \int_{\partial {E}} (r^{-1}|x|-1)^2 \  d\sigma(x)\right) \\
 & \leq C\frac{r}{\e} \int_{\partial {E}} (r^{-1}|x|-1)_+^2 d\sigma(x).
 \end{align*}
 The rest of the proof is the same as the argument for \eref{outerbd} above, choosing as before $\e^2 =\frac{1}{r^{N-1}}\int_{\partial E} (r^{-1}|x|-1)_+^2 d\sigma(x)$.
\end{proof}

\section{Exponential convergence to equilibrium conditional on regularity}\label{sec: exponential rate}
In this final section we discuss the application of our quantitative stability result to the long time behavior of the contact angle motion problem \eref{dropletprob}.  We recall the problem,
\begin{equation}\label{e.dropletprob2}
\left\{
\begin{array}{lll}
-\Delta u(x,t) = \lambda(t) & \hbox{ in } & \Omega_t(u) = \{u(\cdot,t)>0\} \vspace{1.5mm}\\
\tfrac{\partial_t u}{|Du|} = |Du|^2-1 & \hbox{ on } & \Gamma_t(u) = \partial \Omega_t(u),
\end{array}\right.
\end{equation}
where $\lambda(t)$ is a Lagrange multiplier enforcing the volume constraint,
\[ \int u(\cdot,t) \ dx = \text{Vol} \ \hbox{ for all } \ t>0.\]
We compute the energy decay estimate which was stated in the introduction,
\begin{align*}
\frac{d}{dt}\mathcal{J}(\Omega_t) &= \int_{\Omega_t} 2Du\cdot Du_t+ \int_{\Gamma_t} (|Du|^2+1)(|Du|^2-1) \\
&= -2\lambda(t)\int_{\Omega_t} u_t +\int_{\Gamma_t} 2u_t Du\cdot n+(|Du|^2+1)(|Du|^2-1) \\
&= -2\lambda(t) \frac{d}{dt}\left(\int_{\Omega_t} u\right)+\int_{\Gamma_t} -2 (|Du|^2-1)|Du|^2+(|Du|^2+1)(|Du|^2-1) \\
&= -\int_{\Gamma_t} (|Du|^2-1)^2 \leq 0.
\end{align*}
Thus we have obtained,
\begin{equation}\label{e.energyest}
\frac{d}{dt}(\mathcal{J}(\Omega_t) - \mathcal{J}(B_{r_*}))= -\int_{\Gamma_t} (|Du|^2-1)^2,
\end{equation}
or in integrated form,
\begin{equation}\label{e.intform}
\mathcal{J}(\Omega_t) - \mathcal{J}(B_{r_*})= \mathcal{J}(\Omega_0) - \mathcal{J}(B_{r_*}) - \int_0^t\int_{\Gamma_s} (|Du|^2-1)^2 \ d\sigma ds.
\end{equation}
The problem with using this estimate directly, at least in $N \geq 3$, is that the stability result Theorem~\ref{thm: main stability} controls the measure difference squared $|\Omega_t \Delta B_{r_*}|^2$ by the energy dissipation and it is not clear whether $|\Omega_t \Delta B_{r_*}|^2$ controls the energy gap.  Due to this issue we take a different approach, applying the Gr\"{o}nwall argument directly to $|\Omega_t \Delta B_{r_*}|^2$.

For this we need the optimal quantitative Faber-Krahn inequality proven recently by Brasco, De Philippis and Velichkov \cite{Brasco:2015aa}.  Basically this inequality gives the sharp lower bound quadratic growth of the energy $\mathcal{J}(\Omega)$ near it's minimum in terms of the $L^1$ distance.
\begin{theorem}[Brasco, De Philippis, Velichkov]\label{thm: faber-krahn}
There exists a positive constant $c_N$ depending only on dimension such that for every open set $\Omega \subset \real^N$ with finite measure and any ball $B$,
\[ |\Omega|^{\frac{2}{N}+1}\lambda(\Omega) - |B|^{\frac{2}{N}+1}\lambda(B) \geq c_N \textup{Vol}^2 \mathcal{A}(\Omega)^2,\]
where $\mathcal{A}(\Omega)$ is the asymmetry, the infimum over all balls $B \subset \real^N$ of $|\Omega \Delta B|/|B|$.
\end{theorem}
From this Theorem we can easily derive a stability estimate of the capillary energy.
\begin{corollary}\label{cor: fk cor}
For every open set $\Omega \subset \real^N$ with finite measure and $B_r*$ the ball with minimal energy for $\mathcal{J}$,
\begin{equation}\label{e.fkcor}
 (\mathcal{J}(\Omega) - \mathcal{J}(B_{r_*}))^{1/2} \geq c\left(N,\frac{r_*}{|\Omega|^{\frac{1}{N}}}\right)\mathcal{J}(B_{r_*})^{1/2}\frac{|\Omega \Delta B_{r_*}|}{|B_{r_*}|}.
\end{equation}
\end{corollary}
We postpone the proof of the Corollary till the end of the section.  The argument to show that $B_{r_*}$ minimizes the energy $\mathcal{J}$ over all open sets $\Omega$ with finite measure goes as follows.  First let $B_r$ be the ball with volume $|B_r| = |\Omega|$. Then the Polya-Szeg\"{o} principle implies that the Schwarz symmetrization of $u_\Omega$ has the same volume but lower Dirichlet energy than $u_\Omega$ and so 
\[\mathcal{J}(B_r) \leq \mathcal{J}(\Omega).\]  Explicit computation of the radially symmetric solutions, see Appendix~\ref{sec: computations}, shows that $B_{r_*}$ has the minimal energy among all balls.  For these two steps we have separate stability estimates, respectively, the Theorem of \cite{Brasco:2015aa} copied above, and the calculus computation in Appendix~\ref{sec: computations}.

We make several smaller comments about Theorem~\ref{thm: dynamic stability} before we go to the proof.

\begin{remark}
We expect that the convergence modulo translation can be upgraded to convergence to a unique ball $B_{r_*}(x_*)$ using the ideas in \cite{FKdrops} (Proposition $5.2$) with some extra work.
\end{remark}
\begin{remark}
 In $N=2$ it should be possible to get the stability estimate Theorem~\ref{thm: main stability} in Hausdorff distance.  Then one can show a quadratic upper bound on the energy growth near the minimum (in Hausdorff distance) and apply a Gr\"{o}nwall argument directly to the energy.  To go from the exponential convergence of the energy to convergence in measure or in Hausdorff distance one would still need a stability estimate for the Faber-Krahn inequality, although the optimal scaling is not necessary in that case.  
 \end{remark}
Now we prove Theorem~\ref{thm: dynamic stability}, it is very simple given the set up.
\begin{proof}[Proof of Theorem~\ref{thm: dynamic stability}]   We define,
\[ \e(t)^{1/2} =\inf_{x \in \real^N} \frac{|\Omega \Delta B_{r_*}(x)|}{|B_{r_*}|}\]
Instead of trying to use the energy dissipation estimate \eref{intform} directly, we apply the quantitive Faber-Krahn inequality \eref{fkcor} in combination with our stability result Theorem~\ref{thm: main stability} to obtain,
\begin{equation}\label{e.intform2}
r_*^N\e(t) \leq C(\mathcal{J}(\Omega_0) - \mathcal{J}(B_{r_*})) - cr_*^{N-1}\int_0^t \e(t) ds,
\end{equation}
and therefore,
\[\e(t) \leq Cr_*^{-N}(\mathcal{J}(\Omega_0) - \mathcal{J}(B_{r_*}))e^{-cr_*^{-1}t}.\]
\end{proof}
\begin{proof}[Proof of Corollary~\ref{cor: fk cor}]
We carry out the argument described above using the stability estimates,
\[ \mathcal{J}(\Omega) - \mathcal{J}(B_r) = \lambda(\Omega) - \lambda(B_r) \geq c_N\textup{Vol}^2r^{-3N-2}|\Omega \Delta B_r|\]
by Theorem~\ref{thm: faber-krahn}.  By the explicit computation in Appendix~\ref{sec: computations} we have,
\[ \mathcal{J}(B_r) - \mathcal{J}(B_{r_*})  \geq c_N r^{N-2} |r-r_*|^2 \geq c_N r^{-N} |r^N-r_*^N|^2. \]
Thus,
\begin{align*}
 \mathcal{J}(\Omega) - \mathcal{J}(B_{r_*}) &\geq c_N(\textup{Vol}^2r^{-(3N+2)}+r^{-N})|\Omega \Delta B_{r_*}|^2 \\
 &= c_N(r_*^{2N}\textup{Vol}^2r^{-(3N+2)}+r_*^{2N}r^{-N})\left(\frac{|\Omega \Delta B_{r_*}|}{|B_{r_*}|}\right)^2 \\
 &  \geq c_Nf\left(\frac{r_*}{|\Omega|^{1/N}}\right)\mathcal{J}(B_{r_*})\left(\frac{|\Omega \Delta B_{r_*}|}{|B_{r_*}|}\right)^2
 \end{align*}
 where the function $f(s) = s^{3N+2}+s^N$.
\end{proof}

\subsection{Conditions for regularity}\label{sec: conditions} Now we make more precise a set of conditions on the initial data under which the regularity assumed in Theorem~\ref{thm: dynamic stability} is expected to be true.  First we recall the geometric condition introduced in \cite{FKdrops}.  
\begin{definition}
A domain $\Omega$ is said to have the \emph{$\rho$-reflection} property if $B_\rho(0) \subset \Omega$ and for every half space $H \subset \real^N$ which does not intersect $B_\rho(0)$ and the corresponding reflection operator $R$,
\[ \Omega \cap H \subset R(\Omega) \cap H.\]
\end{definition}
The property of $\rho$-reflection is preserved by the flow due to the comparison principle, note that $u(Rx)$ has the same associated Lagrange multiplier as $u$.  This actually requires a bit of work to prove since the comparison required does not have strict ordering at the initial time.  

The $\rho$-reflection property is in a sense a quantified version of the moving planes method.  If $\Omega$ has $\rho$-reflection with $\rho=0$ then $\Omega$ is a ball around $0$.  Sets with $\rho$-reflection satisfy also a strong star-shapedness property as long as $\partial \Omega$ stays away from $B_\rho$.  Precisely,
\[ \sup_{x \in \partial \Omega} |x| - \inf_{x \in \partial \Omega} |x| \leq 4\rho \]
and
\[ \hbox{ $\Omega$ is star-shaped with respect to $B_r(0)$ with } \ r = (\inf_{x \in \partial \Omega} |x|^2 - \rho^2)^{1/2},\]
see \cite{FKdrops} Lemmas $3.23$ and $3.24$.  This is a quantified Lipschitz condition on $\partial \Omega$, and the local Lipschitz constant can be made arbitrarily small if $\rho$ is small.

The $\rho$-reflection property was used to establish the long time existence of viscosity solutions to \eref{dropletprob} in \cite{FKdrops}.
\begin{theorem}[Feldman, Kim]
Suppose that $\Omega_0$ has the $\rho$-reflection property for some $0 \leq \rho < \frac{1}{10}\textup{Vol}^{\frac{1}{N+1}}$, then there exists a global in time viscosity solution of \eref{dropletprob} which has the $\rho$-reflection property for all $t>0$.
\end{theorem}
Any initial data satisfying the assumptions of the above result will preserve the John domain property assumption~\eqref{c} globally in time, it is an easy consequence of strong star-shapedness. We remark that with this regularity property the exponential rate for convergence in measure from Theorem~\ref{thm: dynamic stability} can be upgraded to convergence in Hausdorff distance and convergence of the energy.

The next component is the local regularity from assumptions~\eqref{a} and \eqref{c}.  In analogy to the results of Choi, Jerison and Kim \cite{Choi:2009aa,CJK} on the Hele-Shaw flow we expect that initial data with small Lipschitz constant will be smooth at positive times.  Thus for initial data $\Omega_0$ with $\rho$-reflection $\rho>0$ sufficiently small depending on the dimension, we would expect the free boundary to be globally $C^{1,\alpha}$ in $x,t$.

\appendix
\section{Computations}\label{sec: computations}
We record here several useful computations related to the minimal energy shape.  For a ball of radius $r$ the droplet height profile is given by,
$$u(x) = \tfrac{\lambda(B_r)}{2N}(r^2 - |x|^2).$$
To enforce the volume constraint we require
\begin{align*}
\textup{Vol} &= \int_{B_r}\tfrac{\lambda(B_r)}{2N}(r^2 - |x|^2) dx \\
& = \tfrac{\lambda(B_r)}{2N}\omega_N(1-\tfrac{N}{N+2})r^{N+2} \\
& = \tfrac{\lambda(B_r)}{N(N+2)}\omega_Nr^{N+2}.
\end{align*}
This determines the Lagrange multiplier,
\[ \lambda(B_r) = \frac{N(N+2)}{\omega_Nr^{N+2}}\textup{Vol}.\]
From this we can compute,
\[ \frac{d^2}{dr^2} \mathcal{J}(B_r) = \frac{d^2}{dr^2}(\lambda(B_r)\textup{Vol} + |B_r|) = \frac{N(N+2)^2(N+3)}{\omega_N}\frac{\textup{Vol}}{r^{N+4}} + N(N-1)\omega_N r^{N-2},\]
From which we have,
\[ \frac{d^2}{dr^2} \mathcal{J}(B_r) \geq N(N-1)\omega_N r^{N-2}\]
which we use for the stability estimate.

\medskip

We continue with computing $r_*$.  In order that $|Du| = 1$ on $\partial B_{r_*}$ we must have,
\[ r_* = N/\lambda(B_{r_*}) = \frac{\omega_Nr_*^{N+2}}{(N+2)}\textup{Vol}^{-1},\]
which determines the optimal radius,
\[ r_*^{N+1}  = \omega_N^{-1}(N+2)\textup{Vol}\]
Also we see,
\[\lambda(B_{r_*}) = N\left(\tfrac{\omega_N}{N+2}\right)^{\frac{1}{N+1}}\textup{Vol}^{-\frac{1}{N+1}}.\]
From here we can calculate the minimal energy,
\begin{align*}
 \mathcal{J}(B_{r_*}) &= \lambda(B_{r_*})\textup{Vol}+\omega_Nr_*^N \\
 &= N\left(\tfrac{\omega_N}{N+2}\right)^{\frac{1}{N+1}}\textup{Vol}^{\frac{N}{N+1}}+ \omega_N\left(\tfrac{N+2}{\omega_N}\right)^{\frac{N}{N+1}}\textup{Vol}^{\frac{N}{N+1}} \\
 & = \omega_N^{\frac{1}{N+1}}(N+2)^{-\frac{1}{N+1}}\left(2N+1\right)\textup{Vol}^{\frac{N}{N+1}}
 \end{align*}

\bibliographystyle{plain}
\bibliography{DropletRatesArticles}

\begin{thebibliography}{10}

\bibitem{Aftalion:1999aa}
Amandine Aftalion, J{\'e}r\^ome Busca, and Wolfgang Reichel.
\newblock Approximate radial symmetry for overdetermined boundary value
  problems.
\newblock {\em Adv. Differential Equations}, 4(6):907--932, 1999.

\bibitem{Alexandrov}
A.D. Alexandrov.
\newblock A characteristic property of spheres.
\newblock {\em Annali di Matematica Pura ed Applicata}, 58(1):303--315, 1962.

\bibitem{BNST08}
B~Brandolini, C~Nitsch, P~Salani, and C~Trombetti.
\newblock On the stability of the {S}errin problem.
\newblock {\em J. Differential Equations}, 245(6):1566--1583, 2008.

\bibitem{BNST-alt}
B.~Brandolini, C.~Nitsch, P.~Salani, and C.~Trombetti.
\newblock Serrin-type overdetermined problems: an alternative proof.
\newblock {\em Arch. Ration. Mech. Anal.}, 190(2):267--280, 2008.

\bibitem{BNST09}
B.~Brandolini, C.~Nitsch, P.~Salani, and C.~Trombetti.
\newblock Stability of radial symmetry for a {M}onge-{A}mp{\`e}re
  overdetermined problem.
\newblock {\em Ann. Mat. Pura Appl. (4)}, 188(3):445--453, 2009.

\bibitem{BNST0}
Barbara Brandolini, Carlo Nitsch, Paolo Salani, and Cristina Trombetti.
\newblock A note on the {S}errin problem in the plane.
\newblock {\em Matematiche (Catania)}, 63(2):83--92 (2009), 2008.

\bibitem{Brasco:2015aa}
Lorenzo Brasco, Guido De~Philippis, and Bozhidar Velichkov.
\newblock Faber-{K}rahn inequalities in sharp quantitative form.
\newblock {\em Duke Math. J.}, 164(9):1777--1831, 2015.

\bibitem{CJK}
Sunhi Choi, David Jerison, and Inwon Kim.
\newblock Regularity for the one-phase {H}ele-{S}haw problem from a {L}ipschitz
  initial surface.
\newblock {\em Amer. J. Math.}, 129(2):527--582, 2007.

\bibitem{Choi:2009aa}
Sunhi Choi, David Jerison, and Inwon Kim.
\newblock Local regularization of the one-phase {H}ele-{S}haw flow.
\newblock {\em Indiana Univ. Math. J.}, 58(6):2765--2804, 2009.

\bibitem{Ciraolo:2017aa}
Giulio Ciraolo and Francesco Maggi.
\newblock On the shape of compact hypersurfaces with almost-constant mean
  curvature.
\newblock {\em Comm. Pure Appl. Math.}, 70(4):665--716, 2017.

\bibitem{Ciraolo:2014aa}
Giulio Ciraolo and Rolando Magnanini.
\newblock A note on {S}errin's overdetermined problem.
\newblock {\em Kodai Math. J.}, 37(3):728--736, 2014.

\bibitem{Ciraolo:2016aa}
Giulio Ciraolo, Rolando Magnanini, and Vincenzo Vespri.
\newblock H{\"o}lder stability for {S}errin's overdetermined problem.
\newblock {\em Ann. Mat. Pura Appl. (4)}, 195(4):1333--1345, 2016.

\bibitem{Escher:2015aa}
Joachim Escher and Patrick Guidotti.
\newblock Local well-posedness for a quasi-stationary droplet model.
\newblock {\em Calc. Var. Partial Differential Equations}, 54(1):1147--1160,
  2015.

\bibitem{FKdrops}
William~M. Feldman and Inwon~C. Kim.
\newblock Dynamic stability of equilibrium capillary drops.
\newblock {\em Archive for Rational Mechanics and Analysis}, pages 1--60, 2013.

\bibitem{Figalli:2010aa}
A.~Figalli, F.~Maggi, and A.~Pratelli.
\newblock A mass transportation approach to quantitative isoperimetric
  inequalities.
\newblock {\em Invent. Math.}, 182(1):167--211, 2010.

\bibitem{GlasnerKim09}
K.~Glasner and I.~C. Kim.
\newblock Viscosity solutions for a model of contact line motion.
\newblock {\em Interfaces Free Bound.}, 11(1):37--60, 2009.

\bibitem{MR2144627}
K.~B. Glasner.
\newblock A boundary integral formulation of quasi-steady fluid wetting.
\newblock {\em J. Comput. Phys.}, 207(2):529--541, 2005.

\bibitem{MR2221703}
K.~B. Glasner.
\newblock Variational models for moving contact lines and the quasi-static
  approximation.
\newblock {\em European J. Appl. Math.}, 16(6):713--740, 2005.

\bibitem{glasner}
KB~Glasner.
\newblock A boundary integral formulation of quasi-steady fluid wetting.
\newblock {\em Journal of Computational Physics}, 207(2):529--541, 2005.

\bibitem{Greenspan:1976aa}
H.~P. Greenspan.
\newblock On the deformation of a viscous droplet caused by variable surface
  tension.
\newblock {\em Studies in Appl. Math.}, 57(1):45--58, 1976/77.

\bibitem{greenspan}
HP~Greenspan.
\newblock On the motion of a small viscous droplet that wets a surface.
\newblock {\em J. Fluid Mech}, 84(1):125--143, 1978.

\bibitem{GrunewaldKim11}
N.~Grunewald and I.~Kim.
\newblock A variational approach to a quasi-static droplet model.
\newblock {\em Calculus of Variations and Partial Differential Equations},
  41:1--19, 2011.

\bibitem{Guidotti:2015aa}
Patrick Guidotti.
\newblock Equilibria and their stability for a viscous droplet model.
\newblock {\em Nonlinearity}, 28(9):3175--3191, 2015.

\bibitem{Hurri-Syrjanen:1994aa}
Ritva Hurri-Syrj{\"a}nen.
\newblock An improved {P}oincar{\'e} inequality.
\newblock {\em Proc. Amer. Math. Soc.}, 120(1):213--222, 1994.

\bibitem{MR3627438}
B.~Krummel and F.~Maggi.
\newblock Isoperimetry with upper mean curvature bounds and sharp stability
  estimates.
\newblock {\em Calc. Var. Partial Differential Equations}, 56(2):Paper No. 53,
  43, 2017.

\bibitem{Magnanini:2016aa}
Rolando Magnanini and Giorgio Poggesi.
\newblock On the stability for alexandrov's soap bubble theorem.
\newblock {\em to appear in Journal d'Analyse Math{\'e}matiques}, 2016.

\bibitem{MR3268920}
A.~Mellet.
\newblock The thin film equation with non-zero contact angle: a singular
  perturbation approach.
\newblock {\em Comm. Partial Differential Equations}, 40(1):1--39, 2015.

\bibitem{reilly1974}
Robert~C. Reilly.
\newblock On the hessian of a function and the curvatures of its graph.
\newblock {\em Michigan Math. J.}, 20(4):373--383, 04 1974.

\bibitem{MR0474149}
Robert~C. Reilly.
\newblock Applications of the {H}essian operator in a {R}iemannian manifold.
\newblock {\em Indiana Univ. Math. J.}, 26(3):459--472, 1977.

\bibitem{MR996826}
Antonio Ros.
\newblock Compact hypersurfaces with constant higher order mean curvatures.
\newblock {\em Rev. Mat. Iberoamericana}, 3(3-4):447--453, 1987.

\bibitem{Serrin71}
J.~Serrin.
\newblock A symmetry problem in potential theory.
\newblock {\em Archive for Rational Mechanics and Analysis}, 43:304--318, 1971.

\bibitem{Wang:2009aa}
Xu-Jia Wang.
\newblock The {$k$}-{H}essian equation.
\newblock In {\em Geometric analysis and {PDE}s}, volume 1977 of {\em Lecture
  Notes in Math.}, pages 177--252. Springer, Dordrecht, 2009.

\bibitem{Weinberger:1971aa}
H.~F. Weinberger.
\newblock Remark on the preceding paper of {S}errin.
\newblock {\em Arch. Rational Mech. Anal.}, 43:319--320, 1971.

\end{thebibliography}

\end{document}